\documentclass[11pt, reqno]{amsart}
\usepackage{amsmath, amsthm, a4, latexsym, amssymb}

\def\@rmrk#1#2{\refstepcounter
    {#1}\@ifnextchar[{\@yrmrk{#1}{#2}}{\@xrmrk{#1}{#2}}}

\makeatletter\@addtoreset{equation}{section}\makeatother

 \newfont{\bfit}{cmbxti10 scaled 2000}
 \newfont{\biggi}{cmr12 scaled 2000}

 \newcommand{\R}{\mathbb{R}}
 
 \newcommand{\N}{\mathbb{N}}

 \newcommand{\me}{\mathbb{E}}

 \newcommand{\skrib}{{\mathcal B}}
 \newcommand{\skric}{{\mathcal C}}

 \newcommand{\skrim}{{\mathcal M}}

 \newcommand{\skrin}{{\mathcal N}}

 \newcommand{\skrix}{{\mathcal X}}

 \newcommand{\sfrac}[2]{\mbox{$\frac{#1}{#2}$}}

\def\1{{\mathchoice {1\mskip-4mu\mathrm l}      
{1\mskip-4mu\mathrm l}
{1\mskip-4.5mu\mathrm l} {1\mskip-5mu\mathrm l}}}

\newcommand{\eq}{\begin{equation}}
\newcommand{\en}{\end{equation}}

{\nopagebreak {\hspace*{\fill}\rule{2mm}{2mm}}\\ }

{\nopagebreak {\hspace*{\fill}\rule{2mm}{2mm}}\\ }

\renewcommand{\subsection}{\secdef \subsct\sbsect}
\newcommand{\subsct}[2][default]{\refstepcounter{subsection}
\vspace{0.15cm}
{\flushleft\bf \arabic{section}.\arabic{subsection}~\bf #1  }
\nopagebreak\nopagebreak}
\newcommand{\sbsect}[1]{\vspace{0.1cm}\noindent
{\bf #1}\vspace{0.1cm}}

\newtheorem{theorem}{Theorem}[section]
\newtheorem{lemma}[theorem]{Lemma}
\newtheorem{cor}[theorem]{Corollary}
\newtheorem{prop}[theorem]{Proposition}

\newtheoremstyle{thm}{1.5ex}{1.5ex}{\itshape\rmfamily}{}
{\bfseries\rmfamily}{}{2ex}{}

\newtheoremstyle{rem}{1.3ex}{1.3ex}{\rmfamily}{}
{\itshape\rmfamily}{}{1.5ex}{}
\theoremstyle{rem}

\refstepcounter{subsection}

\def\thebibliography#1{\section*{References}
  \list%
  {\arabic{enumi}.}
    {\settowidth\labelwidth{[#1]}\leftmargin\labelwidth
    \advance\leftmargin\labelsep
    \parsep0pt\itemsep0pt
    \usecounter{enumi}}
    \def\newblock{\hskip .11em plus .33em minus .07em}
    \sloppy                   
    \sfcode`\.=1000\relax}

\begin{document}
\title[Large deviations for spatial telecommunication systems: The boolean mode]
{\Large Large deviations for spatial telecommunication systems: The boolean mode}
	\author[]{}

\maketitle
\thispagestyle{empty}
\vspace{-0.5cm}

\centerline{\sc By  A. K. Boahen$^1$, T. Katsekpor$^{2}$ and  K.  Doku-Amponsah$^{3}$}

{$^1$AIMS  Ghana}

{$^2$Department  of  Mathematics, University of  Ghana, Legon,Accra}

{$^3$Department  of  Statistics and Actuarial  Science, University of  Ghana, Legon,Accra}

{$^3$ Email: kdoku-amponsah@ug.edu.gh}

\vspace{0.5cm}

\renewcommand{\thefootnote}{}
\renewcommand{\thefootnote}{1}

\vspace{-0.5cm}

\begin{quote}{\small }{\bf Abstract.}
Spatial telecommunication systems have evolved along the years, leading to some concerns that telecommunication companies are facing today. The main inquietude is the ability to provide quality service to customers or users in a dense regime. Therefore, questions such as : what is the best possible configurations of base stations and users that maximizes quality service? Is it possible to estimate and control the probability of bad service, which may be seen as a rare event? and many more arise. These questions often involve estimating the tail distribution of events, which falls under the scope of large deviation principles. In this article, we associate with the Boolean model, the  empirical marked measure which will serve  as  a  statistic  for  the  intensity  measure  of  the  Marked  Poisson  Point  Process  of  devices or  users  and  the  empirical  connectivity measure  which  will  serve  as  a  statistic  for  coverage   probability density  of  the  spatial  telecommunication  area. For  these empirical  measures,  prove large deviation principle (LDP) for well-defined empirical measures.

\end{quote}\vspace{0.5cm}


\vspace{0.3cm}

\vspace{0.3cm}

\section{Introduction }
 Throughout the ages, telecommunication has played a very important role in humanity. It has become an integral part of our world. From sending a text message or an email, a WhatsApp video or voice note, or even making a telephone call, one cannot do without any means of communication in general and telecommunication, in particular. While telecommunication took a handful number of forms at the beginning, it has grown at a mind-blowing rate and presents itself under various forms. For example, smoke signals were used in some part of Northern America and India to communicate whereas talking drums was a means of communication in some parts of Africa, with notable example being the Akan tribe  in Ghana\cite{IntroTalkindrums}\cite{IntroTalkindrums2}.
 Telecommunication has evolved in ways that no one would have thought of some centuries ago.\\

The evolution of telecommunication systems has made it more accessible and hence caused a remarkable growth in the number of telecommunications users. The increase in number of users sometimes creates perturbations in the system, which may lead to connectivity issues and poor service. The settings of telecommunications provide mathematicians with a very rich structure. Many researchers have shown interest in the mathematical aspect of telecommunication systems. It has been observed that the spatial configuration of users in a telecommunication area is a very important aspect of the systems. A tool that has been developed and has gained a lot of attention among those whose research interest is in mathematical modelling of telecommunication is stochastic geometry. The area of stochastic geometry has provided researchers with important concept and techniques to study spatial telecommunications.\\

In modelling a telecommunication system, we are mostly concerned with questions of connectivity and coverage area. On one hand, the subject of coverage consists of knowing the extent to which the signal emanating from users can cover the communication area while on the other hand, connectivity addresses the issue of the distance a message can reach in the communication area. Trying to keep pace with the new forms of telecommunications, the most notable being wireless networks, researchers are always trying to find the optimal way in which users of these systems can communicate.  In some settings, the message is sent from one user to the base station, and then relayed from the station to the user it is directed to. In some other settings (which has been adopted in recent years), authors considered messages that travel according to a multihop-functionality, that is from one user to the other until it reaches its destination instead of travelling to a base station before being sent back to the user the message has been destined to. According to \cite{bangerter2014networks}, the multihop-functionality mode of communication is the main idea behind next-generation wireless networks. The precondition in such networks is the ability to provide, on average, quality service to users. For such events where the service quality deteriorates, there is a need to be able to control and estimate the probability of occurrence \cite{hirsch2016large}. The mathematical theory used to accomplish such a task is known as large deviations' theory. \\ 

Large deviations' theory may be seen as a collection of mathematical techniques or tools of a stochastic nature, that are employed in the description or estimation of the asymptotic behavior of extremely rare events. The theory of large deviation finds its applications in many areas, of which the conventional ones are risk management and information theory.\\ 

Doku-Amponsah, in \cite{doku2006large}, defined empirical neighborhood and pair measures for a finite colored graph. He established the large deviation principle (LDP) for these two measures for a class of colored random graphs and their joint large  deviation principle using the technique of change of measure, the method of types and the method of mixtures. He then proceeded to derive the LDP for the degree distribution of the famous Erd\"{o}s-Renyi graphs in the sparse case. Finally, with the established LDP for these measures, Doku-Amponsah demonstrated the asymptotic equipartition properties for hierarchical and networked structures.\\

Hirsch et al., in 2016, proved an LDP for the empirical measure of connectable receivers in a wireless network. They associated to individual transmitters a set of connectable receivers with SINR larger than a fixed threshold. The authors identified the receivers connectable to the origin of the Euclidean space used in modelling the communication area as a Cox point process and proceeded to establish the LDP for the rescaled process of the receivers as the connectivity threshold was approaching zero. In addition to the technical subtleties of their work, the authors were able to reconcile their findings with importance sampling by decreasing the discrepancy in quantifying the probability of rare events. The paper used the Dawson-G\"{a}rtner technique and the contraction principle to achieve its purpose.\cite{hirsch2016large}.\\ 

Umar et.al~ in \cite{umar2017statistical}, studied the joint LDP of the empirical spin and bond measures of the uniformly random $d$-regular graph. He assigned spin values to the vertices of the graph using the Potts model. The author then used the method of types to derive the rate function of the joint LDP. He noticed that the rate function obtained was a sum of two entropies and the typical behavior of the graph was seen when the rate function was equal to zero. The condition for obtaining the typical behavior was achieved when the empirical measure was the same as the product measure of the spin empirical measure and the bond measure.\\

In \cite{LDCapConsRElNet}, the authors, derived "the large deviation principle for the space-time evolution of frustrated transmitters" when the number of transmitters becomes large. The transmitters and relays were set up in a bounded region of $\mathbb{R}^d,d\in \mathbb{N}$, which represented the communication area. The number and position of transmitters were random, whereas the relays were fixed. In their dissertation, a frustrated transmitter is one who selects an occupied relay and using exponential approximation techniques, the authors proved that the family of measure-valued processes representing the proportion of frustrated transmitters satisfied an LDP with good rate function given by the infimum of an entropy.\\

Enoch Sakyi-Yeboah et al, in the paper titled "Large deviation principle, Sharron-McMillan-Brieman theorem for super-critical telecommunication networks", obtained a large deviation asymptotics for super-critical communication networks modelled as a Signal-to-Interference-Noise-Ratio (SINR) network. The authors proved on the scales $\lambda$ and $\lambda^2a_\lambda$, the LDP for two well-defined measures,the power and the empirical connectivity measures, where $\lambda$ was the intensity of a poisson point process (P.P.P) which was used to define the SINR random network.  They proved with ingeniosity the asymptotic equipartition property and the local large deviation principle (LLDP) of the SINR networks and established a classical McMillan theorem for the network. The authors discovered a bound on the cardinality of the space of the SINR networks when the probability of connectivity of the network satisfied some convergence assumptions.  See, \cite{DokuLDShanMcMilSupCritical}.\\

Over the last few years, there has been a keen interest from researchers in studying the optimal way to route messages in a Telecommunication network under the 'famous' Gibbs measure. Most literatures in the field (to the best of my knowledge) base on the well-known SINR (Signal-to-Interference-Noise-Ratio) to derive the Large deviation principles of some well-defined measures, which will therefore help in establishing the Gibbsian structure of the system at hand.  \\

In this work, the motivation for deriving the joint LDP for the empirical connectivity and marked measures in a telecommunication system  is to be able to establish the Gibbsian  structure of the system. In fact, by deriving the rate function of the joint large deviation of the empirical pair and marked measures, the typical behavior of the spatial telecommunication system will be detected. It should be noted that our two measures completely described our model used for the telecommunication system. Therefore, via the rate function, which was expressed as an entropy function, the Gibbs distribution will be derived and the interacting effects of the systems evaluated. The different values of the two measures represent various configurations of users in the communication area and be seen as various thermodynamics states under the Gibbs measure. This research then provide an avenue to find the best configuration of the communication area under the Boolean model that maximizes the service quality and yield the least communication breach.\\ 

Here, we are interested in deriving the joint large deviation principle for the empirical marked and pair measure in a telecommunication system  modelled as Boolean network. Specifically, we will estimate the rate function of the  LDP for  the empirical marked measure,  the conditional pair measure given the empirical marked measure and the joint empirical marked and pair measure in terms of relative entropies.\\

The subsequent parts of this article are structured in the following arrangements. In Section~\ref{sec2}, we provide some mathematical preliminaries necessary for the understanding of the work.Thus, presents the Boolean model and the methods used to establish the large deviation principles of the empirical measures defined in this literature.  Further, in Section~\ref{sec2}, we  also calculate the probability of connection between two users, and we derive the necessary assumptions for convergence.  The  main  results  of  the  article  is  also  presented  in  Section~\ref{sec2}.  In Section~\ref{sec3} we prove  state and  proof  LDP  for  the  empirical  marked  measure.  section~\ref{sec4}  contain  LDP  for the  empirical  connectivity measure    conditional  on  a  given  empirical  mark  measure. Section~\ref{sec5} which is the last Section, draws the conclusions based on the discoveries in the work and identify future research avenues.

\section{MAIN RESULTS}\label{mainresults}\label{sec2}
\subsection{The  Boolean  Model for  Telecommunication  Systems}

Fix  a  dimension  $d\in\N$    and  a  measureable  set  $D\subset \R^d$    with  respect  to  the  Borel-Sigma  algebra  $B(\R^d).$  Denote  by  $Leb$  the  Lebesgues  measure on  $\R^d.$   Let  $\mu_{\lambda}$  be  positive  measure  on  $D$  with  $\mu_{\lambda}(D)=1$.Given  an intensity  measure,   $ \mu_{\lambda}:D \to [0,1]$,   a  coverage probability  density  $Q_{\lambda}$  from  $D$   to  $\R_+$,  we  define  the Boolean network as  follows:

\begin{itemize}
	\item  We  pick  $\mathbb{X}=(X_i)_{i\in I}$   a  Poisson  Point  Process  (PPP)  with  intensity  measure $\mu_{\lambda}: D \to [0,\infty),$ representing the configuration of devices or users in a communication space $D\subseteq \mathbb{R}^d$.
	\item Given   $X,$   we    assign  each  $X_i  $   a random coverage  area  $ B_{X_i}=B_i$  independently  according  to   $Q_\lambda$  the probability distribution of the volume of the balls centered in  $D$.
	\item For   any  two  marked  points  $((X_i,B_i),(X_j, B_j))$  we    connect  an  edge  iff  $ B_{i}\cap B_{j}\not=\emptyset$. 
\end{itemize}

We  consider  $X^{o}(\mathbb{X}, Q_{\lambda})=\Big\{\Big[(X_i,B_{i}), i\in I\Big], \, E\Big\}$  under  the  joint  law  of  the marked P.P.P.  and    the  network.  We  shall  interpret   $X^{o}$    as   a  Boolean  random  network  and   $ X_i^{o}:=(X_i,B_{i})$     the coverage area  of   device $i.$   We  assume  that there  is  a real sequence  $a_{\lambda}$  and  a  functions  $\skric$,  $\mu$   such  that  

$$a_{\lambda}\mu_{\lambda}(x)\to \mu(x) \,\,\mbox{ and}\,\, a_{\lambda}^{-1}Q_{\lambda}(r|x)\to Q(r|x),$$  where   $\lambda a_{\lambda}\to 1$  or  $\lambda a_{\lambda}\to 0$   or  $\lambda a_{\lambda}\to \infty.$   We  shall  restrict  to  the  case  where  $\lambda a_{\lambda}\to 1$. Thus,  we  shall  focus on the  critical Boolen Random  Networks.    We write $\displaystyle\skrib=\Big\{B(x,r)\subseteq D:\,  r\in\R_+,x\in D\Big \}$  and

$$ \displaystyle\skrix=\Cup_{x\in D}\Big\{B(x,r_x):B(x,r_x)\subseteq D\Big \}.$$

By  $\skrim(\skrix\times \skrib)$  we  denote  the  space  of  positive  measures  on  the  space  $\skrix\times \skrib$   equipped  with  $\tau-$  topology. Henceforth,  we shall  refer  to $\skrix$  as  locally  finite  subset  of  the  set  $D $  and    let    $\skrin(\skrib)$  be  the  set  of  counting  numbers  on  $\skrib$   equipped  with  the  discrete  topology.  For any Boolean network $X^{o}$  we  define a probability measure, the
\emph{empirical mark measure}, ~ $L_{1}^{o}\in\skrim(\skrix\times \skrib)$,~by
$$L_{1}^{o}([x,b]):=\frac{1}{\lambda}\sum_{i\in I}\delta_{X_i^{o}}(x,b)$$
and a symmetric finite measure, the \emph{empirical connectivity measure}
$L_2^{o}\in\skrim(\skrix\times\skrib\times \skrix\times\skrib),$ by
$$L_2^{o}(x,b_1,y, b_{2}):=\frac{1}{\lambda}\sum_{(i,j)\in E}\Big[\delta_{(X_i^o,X_j^o)}+
\delta_{(X_j^o,X_i^o)}\Big](x,b_{1},y,b_{2}).$$

Note  that,  while the empirical marked  measure  and   the  empirical  coverage measure   are  probability  measure, total  mass  of  
the empirical pair measure is
$2|E|/\lambda^2$.  Furthermore, $\|L_{1}^{o}\|$   is  the  number  of  devices  connected  to  device  $i$.  

In our model, we assume that each ball is small enough to contain a finite number of users and that the position of one user does not coincide with the other. 
\begin{prop}[Probability of connection]
	Let $\mathbb{X} = \left(X_i\right)_{i\in I},$ a Poisson  Point Process  with intensity measure $\mu_\lambda,$ which represents the configuration of devices or users in a communication space $D\subseteq \mathbb{R}^d$ and $Q_\lambda$ be the probability distribution  of the volume of the balls centered at $X_i.$  
	The probability $p_{\lambda}(x, b_x,y, b_y)$ that two users located at $X_i=x$ and $X_j=y$, $i\neq j$,  with coverage areas $B(x,R_x)$ and $B(y,R_y)$   respectively are connected is given as
	\begin{align*}
p_{\lambda}(x, b_x,y, b_y):= \left[1 - e^{-\mu_\lambda(b_x - b_y)Q_\lambda\left(b_x\right)Q_\lambda\left(b_y\right)}\right],\;\;(x, b_x,y,b_2)\in  \skrix\times \mathcal{B}\times\skrix\times\mathcal{B}, 
	\end{align*}
	where, $b_x = B(x,R_x),$ $b_y = B(y,R_y),$
$ R_x,R_y\in(0,\infty)$ and $b_x - b_y = \left\{z_1-z_2 : z_1\in b_1, z_2\in b_2 \right\}.$
\end{prop}

\begin{proof}
	Let $X_i=x$ and $X_j=y$ be the location of two users in the communication area $D.$ Also, let $B_i := B(X_i,R_i)=b_x$ and $B_j :=B(X_j,R_j)=b_y$ be the coverage areas of these two users. Let $Q_\lambda$ be the probability distribution of the volume of the coverage areas. The two users are connected when their coverage areas intersect. Now, let $p_\lambda(x,b_x,y, b_y)$ be the  probability that two users are connected, given their locations. We have,
\begin{equation}
\begin{aligned}
p_{\lambda}(X_i,B_i,X_j,B_j)&= \mathbb{P}(B(X_i,R_i)\cap B(X_j,R_j)  \neq \emptyset)\\
	&= 1 - \mathbb{P}(B(X_i,R_i)\cap B(X_j,R_j) = \emptyset)\\
	&=1- e^{-\me\left[\mu_\lambda(B_i - B_j)\right]}
\end{aligned}
\end{equation}

	Now, we have:
	\begin{equation}\begin{aligned}
	\mathbb{E}[\mu_\lambda(B_i - B_i)] &=  \mathbb{E}[ \mathbb{E}[\mu_\lambda(B_i - B_j)| B_j]]\\
	&= \mathbb{E}\left[\int \1_{B_i(b_2)} \mu_\lambda(b_2 - B_j) Q_\lambda(db_2) \right]\\
	&= \int \1_{B_j(b_1)}  \int \1_{B_i(b_2)} \mu_\lambda(b_2 - b_1) Q_\lambda(db_2)Q_\lambda(db_1)\\
	&= \int \int \1_{B_j(b_1)} \1_{B_i(b_2)} \mu_\lambda(b_2 - b_1) Q_\lambda(db_2)Q_\lambda(db_1)\\
	&= \mu_{\lambda}(B_i-B_j)Q_\lambda(B_i)Q_\lambda(B_j)
	\end{aligned}
\end{equation}
	
	Therefore, the probability that two users are connected is,
	\begin{align*}
	p_{\lambda}(X_i,B_i,X_j,B_j)&= 1 - e^{-\mu_{\lambda}(B_i-B_j)Q_\lambda(B_i)Q_\lambda(B_j)}.
	\end{align*}
\end{proof}
This result underlines the symmetric nature of this probability measure. In fact, the probability $p_\lambda(x,b_x,y,b_y) = p_\lambda(y,b_y,x, b_x)$ and this reflects the fact that the event "$X_i$ is connected to $X_j$ is the same as the event "$X_j$ is connected to $X_i$." Also, it is clear that as the volume of at least one of the two balls increases, the probability that the two balls are connected increases and when their volume decreases, so does the probability of connection. 
Now,  we  shall henceforth  assume  that  as $\lambda\rightarrow\infty,$ 
$$\lambda^{-3}\mu_\lambda\rightarrow \mu \;\;\;\text{and}\;\;\; \lambda^2Q_\lambda\rightarrow Q. $$ 

Clearly, with these assumptions, we note that $\mu_\lambda Q_\lambda Q_\lambda \approx Q\mu Q/\lambda,$ and the probability of connection is again going to zero if the number of users is increasing in the order of $\lambda^3$ but the strength of the signals emanating from users, which is reflected in the volume of the balls is decreasing at a rate $\lambda^2.$

We then perform a Taylor expansion around $0$  to obtain the  expression;

\begin{align*}
p_{\lambda}(X_i,B_i,X_j,B_j) &= 1 - \left[1-\mu_{\lambda}(B_i-B_j)Q_\lambda(B_i)Q_\lambda(B_j)+ O\left( \dfrac{1}{\lambda^2}\right)\right]\\
&= \mu_{\lambda}(B_i-B_j)Q_\lambda(B_i)Q_\lambda(B_j) + O\left( \dfrac{1}{\lambda^2}\right).
\end{align*}

Now, multiplying through by $\lambda,$ we get   
\begin{align*}
\lambda p_{\lambda}(X_i, B_i,X_j,B_j) &= \lambda\mu_{\lambda}(B_i-B_j)Q_\lambda(X_i)Q_\lambda(B_j) + O\left( \dfrac{1}{\lambda}\right).
\end{align*}

Hence, taking  the  limit  as  $\lambda$  approaches  $\infty$  we  get 
\begin{align*}
\lambda p_{\lambda}(X_i,B_i,X_j,B_j)\rightarrow\mu(B_i-B_j)Q(B_i)Q(B_j):=\Psi(B_i,B_j),
\end{align*}

which proves  the  proposition.
This implies that the asymptotic  probability density modelling two users that are connected is given by the product of the probability distribution of the respective balls representing the coverage areas of the users and the distribution of the number of users in the intersection of the two coverage areas.

\begin{prop}
	\label{lem:first}
	Let $A_1,A_2,\cdots,A_n$ be a decomposition of $D\times \mathcal{B}\subset\mathbb{R}^d\times\mathcal{B}.$ Let $n<\lambda,$ since a Marked Poisson Point Process is locally finite, then the probability distribution of the measure $L_1^0$ is bounded in the following manner:
	\begin{align}
	\sum_{i=1}^n\log \left[\dfrac{e^{-\mu_\lambda\otimes Q_\lambda(A_i) - \epsilon}\left( \mu_\lambda\otimes Q_\lambda(A_i) - \epsilon\right)^{\lambda \eta(A_i)} }{ \left( \lambda\eta (A_i)\right)!}\right] \leq \log P(L^0_1 = \eta)\nonumber \\
	\leq \sum_{i=1}^n\log \left[\dfrac{e^{-\mu_\lambda \otimes Q_\lambda(A_i) + \epsilon}\left( \mu_\lambda\otimes Q_\lambda(A_i) + \epsilon\right)^{\lambda \eta(A_i)} }{ \left( \lambda\eta (A_i)\right)!}\right] + \eta_n,\label{1}
	\end{align}
	where, $$\lim_{n\rightarrow\infty}\lim_{\lambda\rightarrow\infty}\dfrac{1}{\lambda}\eta_n\left(\lambda,A_1,\cdots,A_n) \right) = 0 $$ and the product measure is defined as
	$$\mu_\lambda\otimes Q_\lambda(x, b_x) = \mu_\lambda (x) Q_{\lambda}(b_X).$$
\end{prop}

\subsection { Main  Results}

\begin{theorem}
	\label{thm:main}
	
	Let $\mathbb{X} = \left(X_i\right)_{i\in I},$ be a Poisson  Point Process  with intensity measure $\mu_\lambda,$ which represents the configuration of devices or users in a communication space $D\subseteq \mathbb{R}^d$ and $Q_\lambda$ be the probability distribution  of the volume of the balls centered  on the  points  $\mathbb{X}.$ i.e. coverage  probability distribution.
	Suppose  that  the  intensity measure $\mu_\lambda : D \rightarrow[0, \infty)$  satisfies $\lambda^{-3}\mu_\lambda \rightarrow \mu$ and the coverage probability distribution, $Q_\lambda :\mathcal{B} \rightarrow (0, 1)$, satisfying $\lambda^2 Q_\lambda \rightarrow Q$. Then, as  $\lambda\to\infty$, the  pair  $(L_1^0 ,L_2^0)$  satisfies  a  large  deviation  principle  in $\mathcal{M}(\skrix\times\skrib)\times \mathcal{M}(\skrix\times\skrib\times\skrix\times \skrib)$ with  speed  $\lambda$  and  rate  function

	\begin{align}
	I(\omega,\pi)&= \begin{cases}
	H(\omega|| \mu\otimes Q)+\dfrac{1}{2} \Big [H(\pi|| \Psi\omega\otimes \omega) + ||\Psi\omega\otimes \omega||-||\pi|| \Big]\;\;\,\, \text{ if\,\,  $||\pi|| <\infty,$}
	\\
	\infty,\;\; \text{otherwise}.
	\end{cases}\label{4}
	\end{align}
	

\end{theorem}

\begin{cor}
	\label{cor:main}
	
	Let $\mathbb{X} = \left(X_i\right)_{i\in I},$ be a Poisson  Point Process  with intensity measure $\mu_\lambda,$ which represents the configuration of devices or users in a communication space $D\subseteq \mathbb{R}^d$ and $Q_\lambda$ be the probability distribution  of the volume of the balls centered  on the  points  $\mathbb{X}.$ i.e. coverage  probability distribution.
	Suppose  that  the  intensity measure is given  by  $\mu_\lambda(dx)=\lambda^3dx$  and the coverage probability distribution is  given  by  $Q_\lambda(db_x)=\sfrac{4}{3}\pi r^3/\lambda^2Vol(D)$. Then, as  $\lambda\to\infty$, the  pair  $(L_1^0 ,L_2^0)$  satisfies  a  large  deviation  principle  in $\mathcal{M}(\skrix\times\skrib)\times \mathcal{M}(\skrix\times\skrib\times\skrix\times \skrib)$ with  speed  $\lambda$  and  rate  function

	\begin{align}
	I(\omega,\varpi)&= \begin{cases}
	H(\omega|| \mu\otimes Q)+\dfrac{1}{2} \Big [H(\varpi|| \Psi\omega\otimes \omega) + ||\Psi\omega\otimes \omega||-||\varpi|| \Big]\;\;\,\, \text{ if\,\,  $||\varpi|| <\infty,$}
	\\
	\infty,\;\; \text{otherwise},
	\end{cases}\label{4}
	\end{align}
	
	where  $$\Psi(b_x,b_y)=\sfrac{16}{9}\pi^2 r_x^3r_y^3\Big[\frac{Vol(b_x-b_y)}{Vol(D)^2}\Big].$$

\end{cor}
Hence,  as  $\lambda\to\infty$  the  average number of  connectivity  per device  behaves   as  $$\|E\|/\lambda\to\sfrac{8\pi^2 }{9Vol(D)^2}\int_{D\times\skrib} \int_{D\times\skrib} r_x^3r_y^3 Vol(b_x-b_y)Q(dbx)Q(dby)dxdy\,\,\, \mbox{ \bf {in  probability}.}$$
We first establish the Large Deviation Principle (LDP) for $L^0_1$  in Section (\ref{LDPMarg}) with the appropriate rate function through the method of types. Using the Gartner-Ellis theorem in Section (\ref{LDPCond}), we prove that the measure $L_2^0$ conditional on $L_1^0$ obeys an L.D.P, and we obtain the solution of the rate function as a function of Entropy. Finally, we will prove the LDP for the joint distribution of $L^0_1$ and $L^0_2$ in Section (\ref{LDP3}) by employing the method of mixtures.

\section{Large deviation principle for the empirical marked  measure $L_1^0.$}\label{LDPMarg}\label{sec3}

In this section, we use the method of types to establish the LDP for the empirical marked measure $L^0_1 = \omega_\lambda$ as $\lambda\rightarrow\infty.$  We  begin  by  stating  the  large deviation  probability   of  $L^0_1 = \omega_\lambda$  below:

\begin{lemma}	\label{thm:First}
	Suppose $X^o$ 
	is a Boolean random network with intensity measure $\mu_\lambda : D \rightarrow[0, \infty)$ that satisfies $\lambda^{-3}\mu_\lambda \rightarrow \mu$ and a coverage probability density $Q_\lambda :\mathcal{B} \rightarrow (0, 1)$, satisfying $\lambda^2 Q_\lambda \rightarrow Q$. Then, we  have
	\begin{equation}\label{eq3.1}
	e^{-\lambda H\left(\eta^{(n)} || \mu^{(n)}\otimes Q^{(n)}\right) + \gamma_1(\lambda)}\leq P(L^0_1 = \eta) \leq   e^{-\lambda H\left(\eta^{(n)} || \mu^{(n)}\otimes Q^{(n)}\right) + \gamma_2(\lambda)}
	\end{equation}
	$\lim_{\lambda\rightarrow\infty} \gamma_1(\lambda) = 0,\;\;\;\lim_{\lambda\rightarrow\infty}\gamma_2(\lambda) =\lim_{\lambda\rightarrow\infty}\frac{1}{\lambda}\eta_n(\lambda,A_1,\dots,A_n),$
	where $\eta^{(n)}$ and $\mu^{(n)}\otimes Q^{(n)}$ are the coarsening projections of $\eta$ and $\mu\otimes Q$ on the decomposition $(A_1,\dots,A_n).$
\end{lemma}
To start the proof for Theorem \ref{thm:First}, we will make use of Stirling's formula:

\begin{proof}
	Assume $\lambda$ is large. We have from the upper bound in Equation \eqref{1}
	\begin{align*}
	\log P(L_1^0 = \eta) &\leq \sum_{j=1}^n\left\{ {-\lambda\mu\otimes Q(A_j)} + \lambda\eta(A_j)\log\left[\lambda\mu\otimes Q(A_j)\right] - \log\left[\lambda\eta(A_j)\right]!\right\}  + \eta_n(\lambda,A_1,\dots,A_n)
	\end{align*}
	Using the upper bound of equation \eqref{eq3.1}, we get,
	\begin{align*}
	\log P(L_1^0 = \eta) &\leq \sum_{j=1}^n \left\{-\lambda\mu\otimes Q(A_j) - \log\left[(2\pi)^{\frac{1}{2}}(\lambda\eta(A_j))^{\lambda\eta(A_j)+\frac{1}{2}}e^{-\lambda\eta(A_j)}\right] \right\} \\
	&\;\;\;\;\;+\sum_{j=1}^n\left\{\dfrac{1}{12\lambda\eta(A_j) + 1} + \lambda\eta(A_j)\log\left[ \lambda\mu\otimes Q(A_j)\right]\right\} + \eta_n(\lambda,A_1.\dots,A_n)\\
	\log P(L_1^0 = \eta)&\leq\sum_{j=1}^n \left\{-\lambda\mu\otimes Q(A_j) -\dfrac{1}{2}\log(2\pi) - \left[\lambda\eta(A_j) + \frac{1}{2}\right]\log\left[\lambda\eta(A_j)\right] +\lambda\eta(A_j)\right\}\\
	&\;\;\;\;\; +\sum_{j=1}^n\left\{ \dfrac{1}{12\lambda\eta(A_j) + 1} + \lambda\eta(A_j)\log\left[\lambda\mu\otimes Q(A_j)\right] \right\} + \eta_n(\lambda,A_1,\dots,A_n)\\
	\log P(L_1^0 = \eta) &\leq \sum_{i=1}^n \left\{ -\lambda\left[\mu\otimes Q(A_j) - \eta(A_j)\right] - \lambda\eta(A_j)\log\dfrac{\eta(A_j)}{\mu\otimes Q(A_j)} -\dfrac{1}{2}\log[\lambda\eta(A_j)]\right\}\\
	&\;\;\;\;\;\;+ \sum_{j=1}^n\left\{\dfrac{1}{12\lambda\eta(A_j) + 1} - \dfrac{1}{2}\log(2\pi) \right\} + \eta_n(\lambda,A_1,\dots,A_n)\\
	\log P(L_1^0 = \eta)&\leq \sum_{i=1}^n \left\{ -\lambda\left[\mu\otimes Q(A_j) - \eta(A_j)\right] - \lambda\eta(A_j)\log\dfrac{\lambda\eta(A_j)}{\mu\otimes Q(A_j)} \right\}\\
	&\;\;\;\;\;\;- \sum_{j=1}^n\lambda\left\{\dfrac{\log[\eta(A_j)]}{2\lambda} - \dfrac{1}{12\lambda^2\eta(A_j) + \lambda} + \dfrac{\log(2\pi)}{2\lambda} + \dfrac{\eta_n(\lambda,A_1,\dots,A_n)}{\lambda} \right\} 
	\end{align*}
	We then choose $\gamma2(\lambda)$ as
	\begin{align*}
	\gamma_2(\lambda)&= \dfrac{\log[\lambda\eta(A_j)]}{2\lambda} - \dfrac{1}{12\lambda^2\eta(A_j) + \lambda} + \dfrac{\log(2\pi)}{2\lambda} +\dfrac{\eta_n(\lambda,A_1,\dots,A_n)}{\lambda}.
	\end{align*}
	We have,
	\begin{align*}
	\lim_{\lambda\rightarrow\infty}\gamma_2(\lambda)&=\lim_{\lambda\rightarrow\infty}\left[ \dfrac{\log[\lambda\eta(A_j)]}{2\lambda} - \dfrac{1}{12\lambda^2\eta(A_j) + \lambda} + \dfrac{\log(2\pi)}{2\lambda} +\dfrac{\eta_n(\lambda,A_1,\dots,A_n)}{\lambda}\right] \\
	&= \lim_{\lambda\rightarrow\infty}\dfrac{1}{\lambda}\eta_n(\lambda,A_1,\dots,A_n),
	\end{align*}
	and this concludes the proof for the upper bound of Theorem (\ref{thm:First}).  Again, suppose $\lambda$ is large, we obtain from the lower bound of equation \eqref{1},
	\begin{align*}
	\log P(L_1^0 = \eta) &\geq \sum_{j=1}^n\left\{ {-\lambda\mu\otimes Q(A_j)} + \lambda\eta(A_j)\log\left[\lambda\mu\otimes Q(A_j)\right] - \log\left[\lambda\eta(A_j)\right]!\right\} .
	\end{align*}
	Using the lower bound of equation \eqref{eq3.1}, we get,
	\begin{align*}
	\log P(L_1^0 = \eta) &\geq \sum_{j=1}^n \left\{-\lambda\mu\otimes Q(A_j) - \log\left[(2\pi)^{\frac{1}{2}}(\lambda\eta(A_j))^{\lambda\eta(A_j)+\frac{1}{2}}e^{-\lambda\eta(A_j)}\right] \right\} \\
	&\;\;\;\;\;+\sum_{j=1}^n\left\{\dfrac{1}{12\lambda\eta(A_j) } + \lambda\eta(A_j)\log\left[ \lambda\mu\otimes Q(A_j)\right]\right\} \\
	\log P(L_1^0 = \eta)&\geq\sum_{j=1}^n \left\{-\lambda\mu\otimes Q(A_j) -\dfrac{1}{2}\log(2\pi) - \left[\lambda\eta(A_j) + \frac{1}{2}\right]\log\left[\lambda\eta(A_j)\right] +\lambda\eta(A_j)\right\}\\
	&\;\;\;\;\; +\sum_{j=1}^n\left\{ \dfrac{1}{12\lambda\eta(A_j) } + \lambda\eta(A_j)\log\left[\lambda\mu\otimes Q(A_j)\right] \right\}\\
	\log P(L_1^0 = \eta) &\geq \sum_{i=1}^n \left\{ -\lambda\left[\mu\otimes Q(A_j) - \eta(A_j)\right] - \lambda\eta(A_j)\log\dfrac{\eta(A_j)}{\mu\otimes Q(A_j)} -\dfrac{1}{2}\log[\lambda\eta(A_j)]\right\}\\
	&\;\;\;\;\;\;+ \sum_{j=1}^n\left\{\dfrac{1}{12\lambda\eta(A_j) } - \dfrac{1}{2}\log(2\pi) \right\} \\
	\log P(L_1^0 = \eta)&\geq \sum_{i=1}^n \left\{ -\lambda\left[\mu\otimes Q(A_j) - \eta(A_j)\right] - \lambda\eta(A_j)\log\dfrac{\eta(A_j)}{\mu\otimes Q(A_j)} \right\}\\
	&\;\;\;\;\;\;- \sum_{j=1}^n\lambda\left\{\dfrac{\log[\lambda\eta(A_j)]}{2\lambda} - \dfrac{1}{12\lambda^2\eta(A_j) } + \dfrac{\log(2\pi)}{2\lambda}  \right\}.
	\end{align*}
	We then choose $\gamma_1(\lambda)$ as
	\begin{align*}
	\gamma_1(\lambda)&= \dfrac{\log[\lambda\eta(A_j)]}{2\lambda} - \dfrac{1}{12\lambda^2\eta(A_j)} + \dfrac{\log(2\pi)}{2\lambda} 
	\end{align*}
	We have,
	\begin{align*}
	\lim_{\lambda\rightarrow\infty}\gamma_1(\lambda)&=\lim_{\lambda\rightarrow\infty}\left[ \dfrac{\log[\lambda\eta(A_j)]}{2\lambda} - \dfrac{1}{12\lambda^2\eta(A_j)} + \dfrac{\log(2\pi)}{2\lambda}\right]= 0
	\end{align*}
	and this concludes the proof for the lower bound of Theorem (\ref{thm:First}).
\end{proof}

\begin{lemma}
	\label{lem:LDP1c}
	Suppose $X^{o}$ is a Boolean random network with intensity measure $\mu_\lambda : D \rightarrow[0, \infty)$ that satisfies $\lambda^{-3}\mu_\lambda \rightarrow \mu$ and a coverage probability density $Q_\lambda :\mathcal{B} \rightarrow (0, 1)$ , satisfying $\lambda^2 Q_\lambda \rightarrow Q$. Then, for large $\lambda,$ we have
	$$|I|\leq 2\lambda\;\;\;\text{almost surely}.$$
\end{lemma}

\begin{proof}
	
		Let $A_1,A_2,\cdots,A_m$ be a disjoint decomposition of $\mathbb{R}^d$  such  that $|\skrix\cap A_k|<a_k<\infty,$ for $a_k(A_k):=a_k\in \R.$  Let  $a=\max(a_1,a_2,a_2,a_3,...,a_m)$  and  observe that, $|I|= \sum_{k=1}^m |I_k|, $ where $I_k=\skrix\cap A_k$.Note  $|I_{1}|,|I_{2}|,|I_{3}|,...,|I_m|$  are independent  Poisson distributed random variables  with mean and variance $\mu(\skrix\cap A_k),$ $k=1,2,3,...,m$,  respectively. Observe that $I_k\leq a,$ for all $k= 1,2,\dots,m$ and hence, by applying the Bennett's inequality to the sequence, $I_1,I_2,\dots,I_m,$ we have
	\begin{align}
	\mathbb{P}\left(|I| - \mathbb{E}|I| > \lambda\right)\leq \exp\{-\frac{\lambda}{a^2}\Phi(a)\},\label{3}
	\end{align}
	
	where $Vol(D)$ means the volume of the geometrical space $D$ and $\Phi(u) = (1+u)\log(1+u) - u.$ Now, using Equation \eqref{3}, we have
	$$P\left(|I|\leq \mathbb{E}|I| + \lambda\right)\geq 1 - \exp\{-\frac{\lambda}{a^2}\Phi(a)\},$$ and this leads to 
	$$\lim_{\lambda\rightarrow\infty}\mathbb{P}\left(|I|\leq 2\lambda\right)\geq 1.$$
	Therefore, $|I|\leq 2\lambda$ almost surely, which ends the proof of Lemma (\ref{lem:LDP1c}).
\end{proof}

\begin{lemma}[LDP for $L_1^0$]
	\label{thm:main1}
	Suppose $X^o$ is a Boolean random network with intensity measure $\mu_\lambda : D \rightarrow[0, \infty)$ that satisfies $\lambda^{-3}\mu_\lambda \rightarrow \mu$ and a coverage probability density $Q_\lambda :\mathcal{B} \rightarrow (0, 1)$ , satisfying $\lambda^2 Q_\lambda \rightarrow Q$. Then, as $\lambda\rightarrow\infty,$ $L_1^0$ obeys an LDP with rate function:
	\begin{align}
	I_1(\omega)&=\begin{cases}
	{H}(\omega||\mu\otimes Q)\;\;\text{if} \;\;||\omega|| =1,\\
	\\
	\infty, \;\;\;\text{otherwise.}
	\end{cases}
	\end{align}\label{5}
\end{lemma}

The proof of the above theorem follows the arguments of \cite{EnochMyProof}.
\begin{proof}
	Let $\mathcal{M}_\lambda\left(\mathcal{X}\times\mathcal{B}\right):= \left\{\eta \in \mathcal{M}\left(\mathcal{X}\times\mathcal{B}\right): \lambda \eta(a) \in \mathbb{N}\;\;\text{for all}\;\; a \in \mathcal{X}\right\} $ and let $S$ be a subset of $\mathcal{M}(\mathcal{X\times B}).$ We write $\beta_n:= \max\left(|\mathcal{X}\times B_1|,|\mathcal{X}\times B_2|,\dots,|\mathcal{X}\times B_n|\right)$ and note that\newline $|\mathcal{X}\times B|<\infty$, for all $i=1,2,\dots,n$ by construction. Using Theorem (\ref{thm:First}) and Lemma (\ref{lem:LDP1c}), we obtain:
	\begin{align*}
	(1 + \lambda+\lambda)^{-n\beta_n}e^{-\lambda \inf_{\{\eta\in F^0\cap \mathcal{M}(\mathcal{X\times B})\}}  H(\eta^{(n)}||\mu^{(n)}\otimes Q^{(n)} + \theta_1(\lambda)} &\leq \sum_{\eta\in F^0\cap \mathcal{M}(\mathcal{X\times B})}  e^{-\lambda H(\eta^{(n)}||\mu\otimes Q^{(n)}) + \theta_2(\lambda)}\\
	&\leq \mathbb{P}(L_1^0 \in F)\\
	&\leq \sum_{\eta\in cl(F)\cap \mathcal{M}(\mathcal{X})}  e^{-\lambda H(\eta^{(n)}||\mu\otimes Q^{(n)}) + \theta_2(\lambda)}\\
	&\leq (1 + \lambda+\lambda)^{n\beta_n}e^{-\lambda K + \theta_2(\lambda)},
	\end{align*}
	where $$K = \inf_{\eta\in cl(F)\cap \mathcal{M}(\mathcal{X\times B})} {H}(\eta^{(n)}||\mu^{(n)}\otimes Q^{(n)}),$$
	$\eta^{(n)}$ and $\mu^{(n)}\otimes Q^{(n)}$ are the coarsening projections of $\eta$ and $\mu\otimes Q$ on the decomposition $(A_1,\dots,A_n).$
	
	Taking the limit as $\lambda\rightarrow\infty,$ we have
	\begin{align*}
	\liminf_{\lambda\rightarrow\infty}\left\{-\inf_{\{\eta\in F^0\cap \mathcal{M}_\lambda(\mathcal{X})\}} \, H(\eta^{(n)}|| \mu\otimes Q^{(n)})\right\}&\leq \lim_{\lambda\rightarrow\infty}\dfrac{1}{\lambda}\log \mathbb{P}\left(L_1^0 \in F\right)\\
	&\leq  \limsup_{\lambda\rightarrow\infty}\left\{-\inf_{\{\eta\in cl(F)\cap \mathcal{M}_\lambda(\mathcal{X})\}} \, H(\eta^{(n)}|| \mu\otimes Q^{(n)})\right\}.
	\end{align*}
	We then observe $cl(F)\cap \mathcal{M}_\lambda(\mathcal{X\times X})\subset cl(F)$ for all $\lambda\in \mathbb{R}_+$ and hence we have
	\begin{align*}
	\limsup_{\lambda\rightarrow\infty}\left\{-\inf_{\{\eta\in cl(F)\cap \mathcal{M}_\lambda(\mathcal{X})\}}\, H(\eta^{(n)}|| \mu\otimes Q^{(n)})\right\}&\leq - \inf_{\{\eta \in cl(F)\}} \, H\left(\eta^{(n)}||\mu^{(n)}\otimes Q^{(n)}\right).
	\end{align*}
	Using the arguments of \cite[page 17]{LDPBible}, we obtain
	\begin{align*}
	\liminf_{\lambda\rightarrow\infty}\left\{-\inf_{\{\eta\in F^0\cap \mathcal{M}_\lambda(\mathcal{X\times B})\}}\, H(\eta^{(n)}|| \mu\otimes Q^{(n)})\right\}&\geq - \inf_{{\eta \in F^0}} \, H\left(\eta^{(n)}||\mu^{(n)}\otimes Q^{(n)}\right).
	\end{align*}
	Therefore, we have
	\begin{align*}
	- \inf_{{\eta \in F^0}}\, H\left(\eta^{(n)}||\mu^{(n)}\otimes Q^{(n)}\right)&\leq \lim_{\lambda\rightarrow\infty}\dfrac{1}{\lambda}\log\mathbb{P}\left(L_1^0 \in F\right)\leq  - \inf_{{\eta \in cl(F)}}\, H\left(\eta^{(n)}||\mu^{(n)}\otimes Q^{(n)}\right),
	\end{align*}
	where $\eta^{(n)}$ and $\mu^{(n)}\otimes Q^{(n)}$ are the coarsening projections of $\eta$ and $\mu\otimes Q$ on the decomposition $(A_1,\dots,A_n).$
	Now, when $n\rightarrow\infty,$ we have
	\begin{align*}
	- \inf_{{\eta \in F^0}}\, H\left(\eta||\mu\otimes Q\right)&\leq \lim_{\lambda\rightarrow\infty}\dfrac{1}{\lambda}\log\mathbb{P}\left(L_1^0 \in F\right)\leq  - \inf_{{\eta \in cl(F)}}\, H\left(\eta||\mu\otimes Q\right),
	\end{align*}
	which concludes the proof of Theorem (\ref{thm:main1}).
\end{proof}

We have established the LDP of the empirical marked measure, $L^0_1.$
We then establish in the following section the LDP of the empirical connectivity measure $L^0_2$ given $L^0_1 = \omega.$

\section{LDP for the empirical connectivity measure $L^0_2$ given $L_1^0.$}\label{LDPCond}\label{sec4}

\begin{lemma}
	\label{thm:Main2}
	Let $\mathbb{X} = \left(X_i\right)_{i\in I},$ be a Poisson  Point Process  with intensity measure $\mu_\lambda,$ which represents the configuration of devices or users in a communication space $D\subseteq \mathbb{R}^d$ and $Q_\lambda$ be the probability distribution  of the volume of the balls centered  on the  points  $\mathbb{X}.$ i.e. coverage  probability distribution. Suppose the probability of connection $p_\lambda(b_X,b_y)$ of two users having their location in  $\mathbb{X}$ satisfies $\lambda p_\lambda(b_X,b_Y) \rightarrow  \mu(b_X - b_Y)Q(b_X)Q(b_Y).$   Then $L^0_2$ conditional on the  event  $\{$ $L^0_1 = \omega_\lambda\}$ satifies a large deviation principle with rate function given as, 
	
	\begin{align}
	I_\omega(\pi)&= \begin{cases}
	\dfrac{1}{2}\Bigg[H(\pi|| \Psi \omega\otimes \omega) + ||\Psi \omega\otimes \omega|| - ||\pi||\Bigg], &\text{if}\;\; ||\pi||<\infty\\
	\\
	\infty,& \text{otherwise.}
	\end{cases}\label{6}
	\end{align}

\end{lemma}

The main method used in this section is the Gartner-Ellis Theorem.  We establish the LDP for the empirical connectivity measure $L^0_2$ given the empirical marked measure $L^0_1 = \omega_{\lambda}$ as $\lambda\rightarrow\infty.$ From the definition of these two measures in Section \ref{sec2}, it is evident that $\lambda L^0_2(b_X,b_Y$ conditional on $L^0_1(b_x) = \omega_{\lambda}(b_X)$ follows a Binomial distribution with parameters, $\dfrac{\lambda^2}{2}\omega_{\lambda}(b_X)\omega_{\lambda}(b_Y)$ and $P_\lambda(b_X,b_Y)$.
\begin{lemma}
	\label{lem:LDP2a}
	Let $\mathbb{X} = \left(X_i\right)_{i\in I},$ a P.P.P with intensity measure $\mu_\lambda,$ represent the configuration of devices or users in a communication space $D\subseteq \mathbb{R}^d$ and $Q_\lambda$ be the probability distribution of the volume of the balls centered at $X_i.$ Suppose  the probability of connection $p_\lambda(b_X,b_y)$ of two users having their location in  $\mathbb{X}$ satisfies $\lambda p_\lambda(b_X,b_Y) \rightarrow  \mu(b_X - b_Y)Q(b_X)Q(b_Y).$ Furthermore, let $L^0_2$ be conditional on the  event  $\{$ $L^0_1 = \omega_\lambda\}$. Now, let $g: \mathcal{B}\times\mathcal{B}\rightarrow \mathbb{R}$ be a bounded function. Then,
	\begin{align}
	\lim_{\lambda\rightarrow\infty}\dfrac{1}{\lambda}\log\Phi_\lambda(g) = \Phi(g)\label{7},
	\end{align}
	where,
	$$\Phi(g) = -\dfrac{1}{2}\int_{b_X\in (D\times\mathcal{B})}\int_{b_Y\in (D\times\mathcal{B})}(1 - e^{g(b_X,b_Y)}) \Psi(b_X,b_Y)\omega_\lambda(db_X)\omega_\lambda(db_Y).$$ 
 
\end{lemma}

\begin{proof}
	We assume $\Phi_\lambda(g)$ is the conditional m.g.f of $L_2^0.$ By definition, we have:
	\begin{align*}
	\Phi_\lambda(g)&= \mathbb{E}\left[e^{\lambda\langle g,  L^2_0 \rangle}|L^0_1 = \omega_\lambda\right]\\
	&= \mathbb{E}\left[e^{\lambda\int\int g(b_X,b_Y) L^2_0(db_X, db_Y)}|L^0_1 = \omega_{\lambda}\right].
	\end{align*}
	
	Now, let $A_1,\dots, A_n$ be a decomposition of $\left(D\times \mathcal{B}\right)^2$ into locally finite subsets. We then write:
	\begin{align*}
	\Phi_\lambda(g)&=  \mathbb{E}\left[\prod_{i=1}^n\prod_{(b_X,b_Y)\in A_j}e^{\lambda g(b_X,b_Y) L^2_0(db_X, db_Y)}|L^0_1 = \omega_{\lambda}\right].
	\end{align*}
	
	But the balls are independent of each other, hence we obtain:
	\begin{align*}
	\Phi_\lambda(g)&=  \prod_{i=1}^n\prod_{(b_X,b_Y)\in A_j} \mathbb{E}\left[ e^{\lambda g(b_X,b_Y) L^2_0(db_X, db_Y)}|L^0_1 = \omega_{\lambda}\right].
	\end{align*}
	
	At this point, we observe that the conditional distribution of $L_2^0$  given $L_1^0 = \omega_{\lambda}$ is binomial, i.e,
	$$L^0_2 | L^0_1 = \omega_{\lambda} \sim Bin\left(\dfrac{\lambda^2}{2}\omega_{\lambda}(b_X)\omega_{\lambda}(b_Y), P_\lambda(b_X,b_Y) \right). $$
	
	Also, the conditional expectation in the previous expression is just the conditional m.g.f of $L_2^0.$ 
	We then write:
	\begin{align*}
	\Phi_\lambda(g)&=  \prod_{i=1}^n\prod_{(b_X,b_Y)\in A_j} \left[1 - p_\lambda(b_X,b_Y) + p_\lambda (b_X, b_Y)e^{g(b_X,b_Y)}\right]^{\frac{\lambda^2}{2}\omega_{\lambda}(db_X)\omega_\lambda(db_Y)}.
	\end{align*}
	
	Now, 
	\begin{align*}
	\log \Phi_\lambda(g) &= \sum_{i=1}^n \int_{(b_X,b_Y)\in A_j}\log\left[1 - p_\lambda(b_X,b_Y) + p_\lambda (b_X, b_Y)e^{g(b_X,b_Y)}\right]^{\frac{\lambda^2}{2}\omega_{\lambda}(db_X)\omega_{\lambda}(db_Y)}\\
	&= \sum_{i=1}^n {\int}_{(b_X,b_Y)\in A_j}\frac{\lambda^2}{2}\omega_{\lambda}(db_X)\omega_{\lambda}(db_Y)\log\left[1 - p_\lambda(b_X,b_Y) + p_\lambda (b_X, b_Y)e^{g(b_X,b_Y)}\right]\\
	&= \dfrac{1}{2}\sum_{i=1}^n\int_{(b_X,b_Y)\in A_j}\log\left[1- p_\lambda\left(b_X,b_Y\right)\left(1 - e^{g(b_X,b_Y)}\right)\right]\lambda^2 \omega_{\lambda}(db_X)\omega_{\lambda}(db_Y).
	\end{align*}
	Using the Taylor expansion of $\log(1-x) = -x -\dfrac{x^2}{2}-\dfrac{x^3}{3}-\cdots$, we get,
	\begin{align*}
	&=\dfrac{1}{2}\sum_{i=1}^n\int_{(b_X,b_Y)\in A_j}\left[- p_\lambda(b_X,b_Y)(1 - e^{g(b_X,b_Y)}) + O\left(\dfrac{1}{\lambda^2}\right)\right]\lambda^2\omega_{\lambda}(db_X)\omega_{\lambda}(db_Y)\\
	\end{align*}
	Diving through by $\lambda,$ we get,
	
	\begin{align*}
	\dfrac{1}{\lambda} \log \Phi_\lambda(g) &= -\dfrac{1}{2}\sum_{i=1}^n\int_{(b_X,b_Y)\in A_j}\left[- \lambda p_\lambda(b_X,b_Y)(1 - e^{g(b_X,b_Y)}) + O\left(\dfrac{1}{\lambda}\right)\right]\omega_{\lambda}(db_X)\omega_{\lambda}(db_Y)\\
	\end{align*}
	
	When $\lambda\rightarrow\infty,$ and with the assumption that $\omega_{\lambda}\rightarrow w$, we obtain:
	\begin{align*}
	\lim_{\lambda\rightarrow\infty}\dfrac{1}{\lambda} \log \Phi_\lambda(g) &= -\dfrac{1}{2}\sum_{i=1}^{\infty}\int_{(b_X,b_Y)\in A_j}(1 - e^{g(b_X,b_Y)}) k(b_X,b_Y)\omega_{\lambda}(db_X)\omega_{\lambda}(db_Y)\\
	&= -\dfrac{1}{2}\int_{(D\times\mathcal{B})^2}(1 - e^{g(b_X,b_Y)}) k(b_X,b_Y)\omega_{\lambda}(db_X)\omega_{\lambda}(db_Y)\\
	&=-\dfrac{1}{2}\int_{b_X\in (D\times\mathcal{B})}\int_{b_Y\in (D\times\mathcal{B})}(1 - e^{g(b_X,b_Y)}) k(b_X , b_Y)\omega_{\lambda}(db_X)\omega_{\lambda}(db_Y)\\
	&=-\dfrac{1}{2}\int_{b_X\in (D\times\mathcal{B})}\int_{b_Y\in (D\times\mathcal{B})}(1 - e^{g(b_X,b_Y)}) k\omega_{\lambda}\otimes \omega_{\lambda}(db_X,db_Y)
	\end{align*}
\end{proof}
which concludes the proof of Lemma (\ref{lem:LDP2a}).

It is clear that the function $\Phi$ is differentiable. Therefore, by the Gartner-Ellis theorem, $L_2^0$ conditional on $\left\{L_1^0 = \omega_{\lambda}\right\}$ obeys a large deviation principle with speed $\lambda$ and the variational formulation of the rate function:
\begin{align}
I_w(\pi)&= \dfrac{1}{2}\sup_{g}\left\{\langle g, \pi\rangle - \Phi(g)\right\}\label{4}
\end{align}

The solution to the rate function is given in terms of entropy as:

\begin{align}
I_w(\pi)&= \begin{cases}
\dfrac{1}{2}\Bigg[\, H(\pi|| kw\otimes w) + ||kw\otimes w|| - ||\pi||\Bigg]& \text{if}\;\;||\pi||<\infty\\
\\
\infty,& \text{otherwise}
\end{cases}\label{8}
\end{align}
where, $$\, H(\pi||\mu) = \sum \pi\log\dfrac{\pi}{\mu}.$$

We shall note that the rate function obtained in Equation \eqref{4} is the Legendre transform and hence is a convex function.

\section{Joint Large deviation principle for the empirical pair and marked measures.}\label{LDP3}\label{sec5}
In this section, the method of mixtures is employed. We follow the arguments made in \cite{EnochMyProof} and \cite{DokuLDShanMcMilSupCritical}.

For $\lambda\in\mathbb{R}_+,$ we define
$$\mathcal{M}_\lambda(\mathcal{X}):= \left\{w\in \mathcal{M}(\mathcal{X}):\lambda w(a)\in\mathbb{N}\;\;\text{for all}\;\;a \in \mathcal{X}\right\},$$
$$\mathcal{M}_\lambda(\mathcal{X\times X}):= \left\{\pi\in \mathcal{M}(\mathcal{X\times X}):\lambda \pi(a,b)\in\mathbb{N}\;\;\text{for all}\;\;(a, b) \in \mathcal{X\times X}\right\}.$$

Again we denote by $\Theta_\lambda:=\mathcal{M}_\lambda(\mathcal{X})$ and $\Theta:=\mathcal{M}(\mathcal{X}).$

Finally, we make these two definitions,
$$ P^{(\lambda)}_{\omega_{\lambda}}\left(\eta_\lambda\right) :=\mathbb{P}\left(L^0_2 = \eta_\lambda | L_1^0 = \omega_{\lambda} \right)$$ and
$$P^{(\lambda)}(\omega_{\lambda}):= \mathbb{P}\left(L_1^0 = \omega_{\lambda}\right),$$
then the joint distribution of $L^0_1$ and $L^0_2$ is the mixture of $ P^{(\lambda)}_{\omega_{\lambda}}$ with $P^{(\lambda)}(\omega_{\lambda}):= \mathbb{P}\left(L_1^0 = \omega_{\lambda}\right)$ defined as
\begin{align*}
d\tilde{P}^{\lambda}(\omega_{\lambda},\eta_\lambda):= dP^{(\lambda)}_{\omega_{\lambda}}\left(\eta_\lambda\right)dP^{(\lambda)}\left(\omega_{\lambda}\right).
\end{align*}

\cite{Bi2004} provides the criteria  for the existence of large deviation principles for mixtures and the goodness of the rate function if individual large deviation principles have been established.

The lemmas below make certain the validity of large deviation principles for the mixtures and for the goodness of the rate function if individual large deviation principles are known.

We note that the family of measures $(P^{\lambda}: \lambda\in (0,\infty))$ is exponentially tight on $\Theta$.

\begin{lemma}
	\label{lem:expoTighT}
	The family of measures $(\tilde{P}: \lambda\in\mathbb{R}_+)$ is exponentially tight on $\Theta\times\mathcal{M}(\mathcal{X\times X}$.
\end{lemma}

\begin{proof}
	Let $|E|$ be the cardinality of the edge set $E$ and $t= 1- (1-e^l)e^{-\eta}.$ Then using Chebyshev's inequality and Lemma (\ref{lem:LDP2a}), we get for sufficiently large $\lambda$, 
	\begin{align*}
	\mathbb{P}\left(|E|\leq \lambda^2l\right)&\leq e^{-\lambda^2l}\mathbb{E}\left[e^{|E|}\right]\\
	&\leq e^{-\lambda^2l}\sum_{i=0}^\infty\sum_{k=0}^i e^k {i\choose k} (e^{-\eta})^k(1-e^{-\eta})^{i-k}\dfrac{e^{-\lambda}\lambda^i}{i!}\\
	&\leq e^{-\lambda
		^2l}e^{-\lambda}e^{t\lambda}.
	\end{align*}
	
	Given $s\in \mathbb{N},$ we choose $s>q$ and observe that for sufficiently large $\lambda$, we have
	\begin{align*}
	\mathbb{P}\left(|E|\leq \lambda^2s\right)\leq e^{-\lambda^2s},
	\end{align*}
	Therefore, we obtain
	\begin{align*}
	\mathbb{P}\left(|E|\leq \lambda^2s/2\right)\leq e^{-\lambda^2s/2},
	\end{align*}
	and Lemma (\ref{lem:expoTighT}) is then proved.
\end{proof}

Recall  the  definition  of  the relative  entropy  and  observe  that  it  is a lower semi-continuous  function.


\begin{lemma}[]\label{equbo2}
	$I_w$  are  lower semi-continuous  functions.
\end{lemma}

\begin{proof}
	
	Note  that  the relative  entropy,   $ H(\pi | kw\otimes w),$ is  a  lower  semi-continuous  function  on  the  space $\mathcal{M}_{\lambda}(\mathcal{X\times X}) $  and  $||kw\otimes w|| - ||\pi||$ is  a  linear  function  in  $\pi$.    As  $I_w$  is  a  function  of  a  relative entropy  plus  a  linear  function,  one  can    conclude  that   	$I_w$ is lower semi-continuous.
	
\end{proof}

Using \cite[Theorem~5(b)]{Bi2004}, the two previous lemmas; Lemma~\ref{lem:expoTighT}, Lemma~\ref{equbo2} and the two  LDPs 
we have proved;
Theorem~\ref{thm:main1}  and  Theorem~\ref{thm:Main2} ensure
that under $(\tilde{P}^\lambda)$ the random variables $(L_1^{0}, L_2^{0})$ satisfy a large deviation principle on
$\mathcal{M}_{\lambda}(\mathcal{ X}) \times \mathcal{M}_{\lambda}(\mathcal{X\times X})$ with good rate function  $I$  which  ends  the  proof of  Theorem~\ref{thm:main}.

\section{Conclusion}\label{sec6}
In this final section, we present the summary, the key findings and a conclusion to the article. Some few recommendations are made which can be the starting point for future research works.\\

The paramount aim of this article is to determine the joint large deviation principle for the empirical pair measure and the empirical marked measure in a telecommunication system using the Boolean model. To be able to achieve this objective, it is of prime importance to achieve some required goals. We, first and foremost, generated the position of users in a communication area via a Poisson point process and provided the boolean model settings. We then, after finding the asymptotic behavior of the probability that two users are connected, proceeded to find the large deviation principle of the marginal law of the empirical marked measure and the conditional law of the empirical pair measure given the marginal law measure. Finally, the method of mixtures was employed to formulate the joint large deviation principles of the empirical measures.\\ 

The rate function is the most important part of the large deviation, as it characterizes rare behaviors that may occur in the telecommunication system and the most likely way an unlikely event will occur. It provides a bound on the probability that such rare events may happen. An example of such events is designated by the term frustration events which involve very bad service quality. The rate function provides a way to understand and control such events. Further, it highlights on how the major components of
the telecommunication system, namely: the users and the coverage area, relate in the whole system. The rate function of the joint LDP of the  empirical marked and connectivity measures $( L^0_1 , L^0_2 )$ is non-negative.\\

This work states that the pair measure $(L^0_1,L^0_2)$ of the spatial telecommunication system  which is represented by a Boolean network obeys a large deviation with speed $\lambda$ and large deviation probability
\begin{align}
P\left[(L^0_1,L^0_2) = (w,\pi)\right]\approx e^{-\lambda I(w,\pi)}\label{8},
\end{align}
where the rate function $I(w,\pi)$ is in terms of relative entropies $\, H(w|| \mu\otimes Q), \, H(\pi|| kw\otimes w)$ and a linear function $||\pi||.$ We establish the fact that in a Boolean network, the large deviation principles of the family of laws $P\left[(L^0_1,L^0_2)\in \Gamma\right]$ with speed $\lambda$ and rate function $I(w,\pi)$ is given by
\begin{align*}
-\inf_{(w,\pi)\in \Gamma^0}I(w,\pi)&\leq \liminf_{\lambda\rightarrow\infty}\dfrac{1}{\lambda}\log P\left[(L^0_1,L^0_2)\in \Gamma\right] \leq \limsup_{\lambda\rightarrow\infty}\dfrac{1}{\lambda}\log P\left[(L^0_1,L^0_2)\in \Gamma\right]\\
&\leq -\inf_{(w,\pi)\in cl(\Gamma)}I(w,\pi)
\end{align*}
During the research, we realized that there is a general category of assumptions for which the probability of connection in our settings converges. Different assumptions will lead to an LDP with a different rate. We may encourage researchers to try various classes of assumptions on this model and observe the various variations of the rate of the LDP.\\

Also, there is a penury of articles using the boolean model to describe the happenings in telecommunication systems and study the interactions in such systems. Researchers may take on the task of populating the literature with state-of-the-art work using Boolean model in spatial telecommunication modelling to be able to study the Gibbs measure in this setting.\\

Finally, based on the rate function obtained, a maximization problem can then be solved involving the Gibbs distribution of this system to detect the best configuration of the spatial telecommunication system setting used in this work that minimizes the breach in communication.\\

{\bf \Large Conflict  of  Interest}

The  authors  declare  that they  have  no  conflict  of  interest.\\ 



\begin{thebibliography}{WWW98}


\bibitem[1]{StochGeometryBaccelli}
{\sc Baccelli, F. and  Blaszczyszyn, B.}(2009).
\newblock{Stochastic Geometry and Wireless Networks},
\newblock{Vol. I,Theory Now Publishers Inc.}
\smallskip

\bibitem[2]{StochGeometryBaccelli}
{\sc Baccelli, F. and  Blaszczyszyn, B.}(2009).
\newblock{Stochastic Geometry and Wireless Networks},
\newblock{Vol. II,Applications Now Publishers Inc.}	
\smallskip
	
\bibitem[3]{Bi2004}	
{\sc Biggins, J. D.} (2004).
\newblock{Large deviations for mixtures} 
\newblock In {\em Electron. Comm. Probab.9 60–71.}
\smallskip


\bibitem[4]{bangerter2014networks}
{\sc Bangerter, B.,  Talwar, S., Arefi, R. and Stewart,K.}(2014)
\newblock{Networks and devices for the 5G era,}
\newblock{ 52(2),90-96.},IEEE.
\smallskip
			
\bibitem[5]{Doku2016JointLDColGraph}
{\sc Doku-Amponsah, K.}(2016).
\newblock{Joint large deviation result for empirical measures of the coloured random geometric graphs,
\newblock  {\em SpringerPlus 5,1140.}
\smallskip	
	
\bibitem[6]{doku2006large}
{\sc Doku-Amponsah, K.}(2006).
\newblock{Large deviations and basic information theory for hierarchical and networked data structures,}
\newblock {Ph.D Thesis, University of Bath (UK),ProQuest Dissertations Publishing,2006.}
\smallskip	
			
\bibitem[7]{TheoryofPPDJDaley}
{\sc Daley, D.J. and Vere-Jones,D.}(2006).
\newblock {I: Elementary Theory and Methods, Second Edition,Springer-Verlag, New York,  ISBN : 978-0-387-21337-8.}
\smallskip	

\bibitem[8]{LDPBible}
{\sc A.~Dembo} and {\sc O.~Zeitouni.}(1998)
\newblock Large deviations techniques and applications.
\newblock Springer, New York, 2nd ed. 1998.
\smallskip

\bibitem[9]{AdvancedTopWireNetw}
{\sc Haenggi, M.}(2010).
\newblock{Advanced Topics in Random Wireless Networks}
\newblock{Course notes.}
\smallskip
	
\bibitem[10]{LDCapConsRElNet}
{\sc Hirsch, C.,  Jahnel, B. and Patterson, R.}(2018).
\newblock{Space-time large deviations in capacity-constrained relay networks}.
\smallskip{ALEA, Lat. Am. J. Probab. Math. Stat. 15, 587–615}


\bibitem[11]{hirsch2016large}
{\sc Hirsch, C.m, Jahnel, B.,  Keeler, P. and Patterson, R.}(2016).
\newblock{Large deviation principles for connectable receivers in wireless networks},
\newblock{Advances in Applied Probability 48(4),1061-1094.}
\smallskip

		
\bibitem[12]{WolfgangJahnProbMethodsTelcom}
{\sc}Jahnel, B. and König,W. }(2020).
\newblock Probabilstic  Methods  in  Telecommunication.
\newblock  {Birkhäuser Basel,  Compact  Textbooks in  Mathematics}
\smallskip
	
	
	
	\bibitem[13]{MixModelMPP}
	{\sc Matthew, A. T. and Kottas, A.}(2012).
	\newblock{Mixture Modeling for Marked Poisson Processes},
	\newblock{Bayesian Analysis}, 7(20,{335-362}.
\smallskip

	
\bibitem[14]{IntroTalkindrums2}
{\sc Nketia, J.H.K.}
\newblock{Drumming in Akan Communities of Ghana},
\newblock{Edinburgh,Published online by Cambridge University Press:  22 January 2009}
\smallskip

\bibitem[15]{IntroTalkindrums}
{\sc Locke,  D.}(1990)
\newblock{Drum Damba: Talking Drum Lessons},
\newblock{Crown Point,White Cliffs Media Co; First Edition (January 1, 1990)}

\bibitem[16]{Stirling}	
{\sc Robins, H.}(1995).
\newblock{A Remark on Stirling's Formula}
\newblock{The American Mathematical Monthly 62(1), 26–29.}
\smallskip
	


	

\bibitem[17]{EnochMyProof}
{\sc Sakyi-Yeboah,E. and Doku-Amponsah,K. and Asiedu,L.}(2020).
\newblock{Local large deviation principle, large deviation principle and information theory for the signal-to- interference-plus-noise ratio graph models,}
\newblock{\emph{Journal of Information and Optimization Sciences}, 42:1, 249-273 }  
\smallskip
	
	\bibitem[18]{SKAD2021}
	{ Sakyi-Yeboah, E. and  Kwofie, C. and Asiedu, L. and Doku-Amponsah, K.}(2021).
	\newblock{Large deviation principle for empirical SINR measure of critical telecommunication networks,}
	\newblock {\emph{Journal of Information and Optimization Sciences} 42(1),287-301.}
\smallskip
		
\bibitem[19]{DokuLDShanMcMilSupCritical}
 {\sc Sakyi-Yeboah, E.,  Andam, P.,   Asiedu, L. and Doku-Amponsah, K.}(2020).
\newblock{Large Deviations, Sharron-MacMillan-Breiman Theorem for Super-Critical Telecommunication Networks}
\newblock{ To  appear  \emph{Journal of Information and Optimization Sciences }}.
	
\smallskip
	




\bibitem[21]{umar2017statistical}
{\sc Umar, I. , Lotis,A.and Doku-Amponsah,K}(2021).
\newblock{From Statistical Mechanics to Large Deviations of Uniformly Random D-Regular Graphs,}
\newblock{To  appear \emph{ Journal  of  discrete  Mathematics  and  Cythography}} 

	
	
	\end{thebibliography}
\end{document}